\documentclass{amsart}
\usepackage{amsfonts,amssymb,amscd,amsmath,amsthm,txfonts}
\usepackage{accents}
\usepackage{a4wide}
\usepackage[all]{xy}

\newtheorem{Thm}{Theorem}[section]
\newtheorem*{MainThm}{Main Theorem}
\newtheorem{Lem}[Thm]{Lemma}
\newtheorem{Cor}[Thm]{Corollary}

\theoremstyle{definition}
\newtheorem*{Fact}{Fact}
\newtheorem{Facts}[Thm]{Facts}
\newtheorem{Def}[Thm]{Definition}

\theoremstyle{remark}
\newtheorem{Rem}[Thm]{Remark}
\newtheorem{Rems}[Thm]{Remarks}

\numberwithin{equation}{section}

\DeclareMathOperator{\dom}{dom}
\DeclareMathOperator{\eval}{eval}
\newcommand{\al}[1]{\aleph_{#1}}          
\newcommand{\fmax}{f^\textrm{max}}
\newcommand{\gmin}{g^\textrm{min}}
\newcommand{\Bmin}{B^\textrm{min}}

\newcommand{\myc}{c^{\exists}}
\newcommand{\mycfa}{c^{\forall}}

\newcommand{\esm}{\prec}  

\newcommand{\std}[1]{\check{#1}}
\newcommand{\forc}{\Vdash}

\newcommand{\n}[1]{\underaccent{\tilde}{#1}}

\DeclareMathOperator{\POSS}{POSS}
\DeclareMathOperator{\poss}{poss}
\DeclareMathOperator{\VAL}{VAL}
\DeclareMathOperator{\nor}{nor}
\DeclareMathOperator{\minnor}{minnor}
\DeclareMathOperator{\maxnor}{maxnor}
\DeclareMathOperator{\maxsupp}{maxsupp}
\DeclareMathOperator{\maxposs}{maxposs}

\DeclareMathOperator{\val}{val}
\DeclareMathOperator{\half}{half}
\newcommand{\cc}{\mathfrak{c}}
\newcommand{\cd}{\mathfrak{d}}
\newcommand{\cS}{\mathbf{\Sigma}}
\newcommand{\cSp}{\mathbf{\Sigma}_+}
\newcommand{\cK}{\mathbf{K}}
\newcommand{\nugen}{\n\nu^\text{gen}}
\DeclareMathOperator{\trnk}{trunk}
\DeclareMathOperator{\trnklg}{trnklg}
\newcommand{\supp}{\textrm{supp}}
\newcommand{\suppls}{\textrm{supp}^\textrm{ls}}
\newcommand{\norls}{\textrm{nor}^\textrm{ls}}

\begin{document}
\date{2010-02-04}

\title{Creature forcing and large continuum: The joy of halving}
\author[Jakob Kellner]{Jakob Kellner$^\dag$}
\address{Kurt G\"odel Research Center for Mathematical Logic\\
 Universit\"at Wien\\
 W\"ahringer Stra\ss e 25\\
 1090 Wien, Austria}
\email{kellner@fsmat.at}
\urladdr{http://www.logic.univie.ac.at/$\sim$kellner}
\thanks{$^\dag$ supported by European Union FP7 project PERG02-GA-2207-224747
and Austrian science funds FWF project P21651-N13.}
\author[Saharon Shelah]{Saharon Shelah$^\ddag$}
\address{Einstein Institute of Mathematics\\
Edmond J. Safra Campus, Givat Ram\\
The Hebrew University of Jerusalem\\
Jerusalem, 91904, Israel\\
and
Department of Mathematics\\
Rutgers University\\
New Brunswick, NJ 08854, USA}
\email{shelah@math.huji.ac.il}
\urladdr{http://www.math.rutgers.edu/$\sim$shelah}
\thanks{
$^\ddag$
supported by the United States-Israel
  Binational Science Foundation (Grant no. 2002323),
publication 961.}


\begin{abstract}
  For $f,g\in\omega^\omega$ let $\mycfa_{f,g}$ be the minimal
  number of uniform $g$-splitting trees needed to
  cover the uniform $f$-splitting tree, i.e., for every
  branch $\nu$ of the $f$-tree, one of the
  $g$-trees contains $\nu$.
  Let $\myc_{f,g}$ be the dual notion: For every branch $\nu$,
  one of the $g$-trees guesses $\nu(m)$ infinitely often.
  We show that it is consistent that
  $\myc_{f_\epsilon,g_\epsilon}=\mycfa_{f_\epsilon,g_\epsilon}=\kappa_\epsilon$
  for continuum many pairwise different cardinals $\kappa_\epsilon$
  and suitable pairs $(f_\epsilon,g_\epsilon)$.
  For the proof we introduce a new mixed-limit creature
  forcing construction.
\end{abstract}
\maketitle

\section*{Introduction}

We continue the investigation in~\cite{MR2499421} of the following cardinals
invariants:

Let $f,g$ be functions from $\omega$ to $\omega$ such that 
$f(n)>g(n)$ for all $n$ and furthermore $\lim(f(n)/g(n))=\infty$.
An $(f,g)$-slalom is a sequence $Y=(Y(n))_{n\in\omega}$
such that $Y(n)\subseteq f(n)$ and $|Y(n)|\leq g(n)$ for all $n\in\omega$.
A family $\mathcal Y$ of $(f,g)$-slaloms is a $(\forall,f,g)$-cover, if
for all $r\in \prod_{n\in\omega} f(n)$ there is an $Y\in \mathcal Y$ such that
$r(n)\in Y(n)$ for all $n\in\omega$.
The cardinal characteristic $\mycfa_{f,g}$ is defined as the minimal size of a
$(\forall,f,g)$-cover.

There is also a dual notion:
A family $\mathcal Y$ of $(f,g)$-slaloms is an $(\exists,f,g)$-cover, if
for all $r\in \prod_{n\in\omega} f(n)$ there is an $Y\in \mathcal Y$ such that
$r(n)\in Y(n)$ for infinitely many $n\in\omega$.
We define $\myc_{f,g}$ to be  
the minimal size of an $(\exists,f,g)$-cover 

It is easy to see that $\al0<\myc_{f,g}\leq \mycfa_{f,g}\leq 2^{\al0}$.

Answering a question of Blass related to~\cite{MR1234278}, Goldstern and the
second author~\cite{MR1201650} showed how to force $\al1$ many different values
to $\mycfa_{f,g}$. More specifically, assuming CH and given a sequence
$(f_\epsilon,g_\epsilon,\kappa_\epsilon)_{\epsilon\in\al1}$ of natural
functions $f_\epsilon,g_\epsilon$ with ``sufficiently different growth rate''
and cardinals $\kappa_\epsilon$ satisfying
$\kappa_\epsilon^{\al0}=\kappa_\epsilon$, there is a cardinality preserving
forcing notion that forces $\mycfa_{f_\epsilon,g_\epsilon}=\kappa_\epsilon$ for
all $\epsilon\in\al1$.
In~\cite{MR2499421} we additionally forced
$\myc_{f_\epsilon,g_\epsilon}=\mycfa_{f_\epsilon,g_\epsilon}=\kappa_\epsilon$.

In this paper, we improve%
\footnote{%
  Note that once we have $\al1$ many different cardinals between $\al0$ and the
  continumm, then the continumm has to be much bigger than $\al1$.
}
this result to {\em continumm many} characteristics
$\myc_{f_\epsilon,g_\epsilon}=\mycfa_{f_\epsilon,g_\epsilon}$ in the extension
(something which is a lot easier for $\mycfa$ only, as it is done
in~\cite{MR2425542}).

So the main theorem is:
\begin{MainThm}\label{thm:uncountable}
  Assume that CH holds, that
  $\mu=\mu^{\al0}$, and that
  $\kappa_\epsilon<\mu$
  satisfies $\kappa_\epsilon^{\al0}=\kappa_\epsilon$
  for all $\epsilon\in \mu$.
  Then there is an $\omega^\omega$-bounding,
  cardinality preserving forcing notion $P$ that forces
  the following: $2^{\al0}=\mu$, 
  and there are functions $f_\epsilon,g_\epsilon$
  for $\epsilon\in\mu$ such that
  $\myc_{f_\epsilon,g_\epsilon}=
  \mycfa_{f_\epsilon,g_\epsilon}=\kappa_\epsilon$.
\end{MainThm}
(We can find such $\mu$ and $(\kappa_\epsilon)_{\epsilon\in\mu}$ such that the
$\kappa_\epsilon$ are pairwise different, then we get continuum many pairwise
different invariants in the extension.)

The construction builds on the theory of creature forcing, which is described
in the monograph~\cite{MR1613600} by Ros{\l}anowski and the second author.
However, this paper should (at least formally) be quite self contained 
concerning creature forcing theory; we do however (in~\ref{facts:214f}) cite a 
result of~\cite{MR2499421}.

This paper has two parts: In the first part, we introduce a new creature
forcing construction (to give some ``creature keywords'':
somewhat in between a restricted product and an iteration,
with countable support, basically a lim-inf construction but allowing 
for lim-sup conditions as well).
Using this construction,
we get a much nicer and more general proof of properness compared to the
construction in~\cite{MR2499421}.

This construction (actually a simple case, in particular a pure lim-inf case
without downwards memory) is used the second part to construct the required
forcing. It turn out that we can use very similar proofs to the ones
in~\cite{MR2499421} to show that the furcing notion constructed this way
actually does what we want.

\section{The creature forcing construction}

\subsection{The basic definitions}

\begin{Def}
  Let $I^*$ be some (index) set, and for each
  $i\in I^*$ and $n\in\omega$ fix a finite set
  $\POSS^*_{{=}n,\{i\}}$.
  \\
  For $u\subseteq I^*$ and $n\in\omega$  we set
   \[
     \POSS_{n,u}=\{\eta:\, \eta\text{ is a function},\ \dom(\eta)=n\times u, 
       \text{ and } \eta(m,i)\in \POSS^*_{{=}m,\{i\}}\text {for all }
       m\in n
     \text{ and }i\in u\}.
   \]
\end{Def}
The name $\POSS$
is chosen because this is the set of possibile trunks of conditions,
see below.

We will use the following notation for restrictions of $\eta\in \POSS_{n,u}$:
For $0\leq m\leq n$ and for $w\subseteq u$ we use $\eta\restriction m\in
\POSS_{m,u}$, $\eta\restriction w\in \eta\in \POSS_{n,w}$ and $\eta\restriction
(m\times w)\in \POSS_{m,w}$ (with the obvious meaning).
We will sometimes identify an $\eta\in \POSS_{n,\{i\}}$, i.e., a function with
domain $n\times \{i\}$, with the according function with domain $n$.

\begin{Def}
   $\VAL_{n,u}$ is the set of functions $\mathbf f: \POSS_{n,u}\to\POSS_{n+1,u}$
   satisfying $\mathbf f(\eta)\restriction n=\eta$ for all $\eta\in\POSS_{n,u}$.
\end{Def}
(This is the set of possible elements of the value-set
$\val(\cc)$ of an $n$-ml-creature, see below.)

\begin{Def}\label{def:creature}
  Fix $n\in\omega$. An $n$-ml-creature parameter $\mathfrak{p}_n$ consists of
  \begin{itemize}
    \item $\cK(n)$, the set of $n$-ml-creatures,
    \item   the functions $\supp$, $\suppls$,
      $\nor$, $\norls$,  $\val$ and $\cS$, all
      with domain $\cK(n)$,
  \end{itemize}
  satisfying the following (for $\cc\in \cK(n)$):
  \begin{enumerate}
    \item $\suppls(\cc)\subseteq \supp(\cc)$ are
      finite%
      \footnote{We will later even require: 
      There is a functions $\maxsupp:\omega\to\omega$ such that
      every $n$-ml-creature $\cc$ satisfies $|\supp(\cc)|<\maxsupp(n)$.}
      subsets of $I^*$. We call $\supp(\cc)$ the support of $\cc$.
    \item $\nor(\cc)$ (called norm) and $\norls(\cc)$ are nonnegative 
      reals.%
      \footnote{%
        More particularly, elements of some
        countable set containing $\mathbb{Q}$ and closed under the 
        functions we need, such as $\ln$ etc.
        We can even restrict $\nor$ and $\norls$  to values 
        in $\mathbb{N}$. However, this sometimes leads to slightly
        cumbersome and less natural definitions.%
     }
    \item $\val(\cc)$ is a nonempty subset of $\VAL_{n,\supp(\cc)}$.
    \\
    For $\eta\in \POSS_{n,\supp(\cc)}$, we set
      $\cc[\eta]\coloneqq\{\mathbf f(\eta):\, \mathbf f\in \val(\cc)\}$.
      So $\cc[\eta]$ is a nonempty subset of $\POSS_{n+1,\supp(\cc)}$, and every
      $\nu\in \cc[\eta]$ extends $\eta$.
    \item\label{item1:wq} $\cS(\cc)$, the set
      of ml-creatures that are stronger than (or: successors of)
      $\cc$, is a  subset of $\cK(n)$ such that
      for all $\cd\in \cS(\cc)$ the following holds:
      \begin{enumerate}
        \item if $\cd'\in\cS(\cd)$, then $\cd'\in\cS(\cc)$
          (i.e., $\cS$ is transitive).
        \item $\cc\in \cS(\cc)$ (i.e., $\cS$ is reflexive).
        \item $\supp(\cd)\supseteq \supp(\cc)$ and
          $\suppls(\cd)\cap \supp(\cc)\subseteq \suppls(\cc)$.
	\item\label{item2:wq} $\cd[\eta]\restriction \supp(\cc)\subseteq
               \cc[\eta\restriction\supp(\cc)]$
               for every $\eta\in \POSS(n,\supp(\cd))$.
      \end{enumerate}
  \end{enumerate}
\end{Def}
Of course, with $\cd[\eta]\restriction \supp(\cc)$ we mean $\{\nu\restriction
\supp(\cc):\, \nu\in\cd[\eta]\}$.

\begin{Rems}\label{rem:simplecase}
  \begin{itemize}
    \item
      ``ml'' stands for ``mixed limit'' (the construction mixes lim-sup
      and lim-inf aspects).
      ``ls'' stands for lim sup; 
      $\suppls$ and $\norls$ will correspont to
      the part of the forcing that corresponds to a lim-sup
      sequence. The objects $\supp$ and $\nor$
      will correspond  to the lim-inf part. 
    \item Our application will be a ``pure lim-inf'' forcing:
      We can completely ignore $\suppls$ and $\norls$, or,
      more formally, we can set 
      $\suppls(\cc)=\supp(\cc)$ and $\norls(\cc)=n$
      for all $n$-ml-creatures $\cc$.
    \item Usually we will also have:
      if $\cd\in\cS(\cc)$ then $\nor(\cd)\leq \nor(\cc)$ and
      $\norls(\cd)\leq \norls(\cc)$, but this is not required for
      the following proofs.
    \item In our application (as well as in other potential applications)
      we will not really use $\val(\cc)$ (i.e.,
      a set of functions $\mathbf f$ each mapping every possible trunk $\eta$
      af height $n$ to one of height $n+1$).
      Instead, we will only need 
      $(\cc[\eta])_{\eta\in\POSS_{n,\supp(\cc)}}$ (i.e.,
      the function that
      assigns to each $\eta$ the 
      (nonempty, finite) set of possible extensions $\cc[\eta]$).

      We can formalize this simplification in our framework
      as the following additional requirement:

      Assume that $\mathbf f\in \VAL_{n,\supp(\cc)}$ is such
      that for all $\eta\in \POSS_{n,\supp(\cc)}$ there is 
      a $\mathbf g\in \val(\cc)$ such that $\mathbf f(\eta)= \mathbf g(\eta)$.
      Then $\mathbf f\in \val(\cc)$.
      Or, in other words: $\mathbf f\in \VAL_{n,\supp(\cc)}$ is in 
      $\val(\cc)$ iff $\mathbf f(\eta)\in \cc[\eta]$
      for all $\eta\in\POSS_{n,u^\cc}$.
    \item We could required the following, stronger property instead 
      of \ref{def:creature}.(\ref{item2:wq}) 
      (however, in the case referred to in the previous item, the two versions
      are equivalent anyway):
      
      For all $\mathbf f\in \val(\cd)$ there is some
      $\mathbf g\in \val(\cc)$ such that 
      for each 
      $\eta\in \POSS_{n,\supp(\cd)}$
      \[
        \mathbf f(\eta)\restriction \supp(\cc)=
        \mathbf g(\eta\restriction \supp(\cc)).
      \]
    \item
      Our application will even have the following property:
      $\cc[\eta]$ is  essentially independent of $\eta$;
      there is no ``downwards memory'', the creature does not look at 
      what is going on below.

      More exactly: We will define $\mathfrak p_n$ in a way so that
      for all $\eta,\eta'$ in $\POSS_{n,\supp(\cc)}$ and 
      $\nu\in\cc[\eta]$  the possibility
      $\eta'\cup (\nu\cap ({n}\times I))$ is in $\cc[\eta']$.
    \item So while the application in this paper only uses a simpler setting,
      we give the proof of properness for the more general setting. The reason
      is that this properness-proof is not more complicated for the general
      case, and we hope that we can use this general case for other
      applications.
  \end{itemize}
\end{Rems}

\begin{Def}
  A forcing parameter $\mathfrak{p}$ is a sequence
  $(\mathfrak{p}_n)_{n\in\omega}$ of 
  $n$-ml-creature parameters. 
  Given such a $\mathfrak{p}$, we define 
  the forcing notion $Q_\mathfrak{p}$:
  A condition $p$ consists of 
  $\trnklg(p)\in\omega$, the
  $n$-ml-creatures $p(n)$ for $n\geq \trnklg(p)$, and an object
  $\trnk(p)$ such that:
  \begin{itemize}
    \item 
      $\supp(p(n))\subseteq \supp(p(n+1))$
      for all $n\geq \trnklg(p)$.
    \item
      We set $\dom(p)\coloneqq\bigcup_{n\in\omega} \supp(p(n))$, and for
      $i\in\dom(p)$ we set
      $\trnklg(p,i)=\min\{n\geq \trnklg(p):\, i\in \supp(p(n))\}$.
    \item $\trnk(p)$ is a function with domain
      $\{(m,i):\, i\in\dom(p),m<\trnklg(p,i)\}$ such that
      $\trnk(p)(m,i)$ is in $\POSS^*_{=m,\{i\}}$.
      For $i\in\dom(p)$, we set
      $\trnk(p,i)=\trnk(p)\restriction \{i\}$ (which we
      identify with a function with domain $\trnklg(p,i)$).
    \item $\liminf_{n\to\infty}\nor(p(n))=\infty$.
    \item For each $i\in \dom(p)$ the set 
      $X=\{\norls(p(n)):\, i\in\suppls(p(n))\}$ is unbounded,
      in other words:
      $\limsup(X)=\infty$. In particular there are 
      infinitely many $i$ with
      $i\in\suppls(p(n))$.
  \end{itemize}
\end{Def}

For better readability, we will write $\supp(p,n)$ instead of $\supp(p(n))$,
and the same for $\nor$ etc. 

Note that $Q_\mathfrak p$ could be empty (for example, if all norms of
ml-creatures are bounded by a universal constant). In the following we will
always assume that $Q_\mathfrak p$ is nonempty.

We still have to define the order on $Q_\mathfrak p$. Before we can do this,
we need another notion: $\poss(p,n)$, the sets of elements of
$\POSS_{n,\dom(p)}$ that are ``compatible with $p$'':
\begin{Def}
  For a condition $p$ (or just an according finite sequence of cratures together
  with a sufficient part of the trunk), 
  we define $\poss(p,n)$ as a subset of $\POSS_{n,\dom(p)}$ 
  be induction on $n$. 
  If $n\leq\trnklg(p)$, then $\poss(p,n)$ contains the singleton 
  $\trnk(p)\restriction (n\times \dom(p))$.
  Otherwise $\poss(p,n)$ consists of those $\nu\in \POSS_{n,\dom(p)}$ such that
  $\nu$ is compatible%
  \footnote{%
  I.e., $\nu(m,i)=\trnk(p)(m,i)$
  for all $m<\min(n,\trnklg(p,i))$.
  }
  with $\trnk(p)$ and such that
  $\nu \restriction \supp(p,n)\in p(n)[\eta\restriction \supp(p,n)]$
  for some $\eta\in \poss(p,n-1)$.
\end{Def}

\begin{Def}
  For $p,q\in Q_{\mathfrak p}$, we set $q\leq p$ if the following holds:
  \begin{itemize}
    \item $\trnklg(q)\geq \trnklg(p)$.
    \item  If $n\geq \trnklg (q)$ then
      \begin{itemize}
        \item  $q(n)\in \cS(p,n)$, 
        \item  $\supp(q,n)\cap \dom(p)=\supp(p,n)$, 
            (This implies:
            $\trnklg(q,i)$ is the maximum of
            $\trnklg(p,i)$ and $\trnklg(q)$ for all $i\in\dom(p)$.)
        \item $\suppls(q,n)\cap \dom(p)\subseteq \suppls(p,n)$. 
      \end{itemize}
    \item $\trnk(q)$ extends $\trnk(p)$ (as function), i.e., 
      $\trnk(q)(m,i)=\trnk(p)(m,i)$ whenever $i\in\dom(p)$ and $m<\trnklg(p,i)$.
    \item $\trnk(q)\restriction(\trnklg(q)\times\dom(p))
      \in \poss(p,\trnklg(q))$.
  \end{itemize}
\end{Def}

\begin{Rem}
     Note that our ml-creatures have an
      ``answer''  $\cc[\eta]$ to all 
      $\eta\in\POSS_{n,\supp(c)}$; so in particular  $p(n)$
      has answers to 
      all $\eta\notin\poss(p,n)$.
      In this respect, our creatures carry a lot of seemingly 
      irrelevant information. This is neccessary, however,
      to allow simple proofs of properness and
      rapid reading: this way
      we can, e.g., start with a condition $p$,
      then increase the trunk to some height $h$, strengthen this new
      condition to some $q$, 
      and then ``merge'' $p$ and $q$, by setting $r(n)=p(n)$ for $n<h$
      and $r(n)=q(n)$ otherwise. This would not be possible if we dropped
      the information about ``impossible'' $\eta\in\POSS_{n,\supp(c)}$
      from the creatures.
\end{Rem}

\begin{Facts}
  \begin{itemize}
    \item Assume that  $p$ is a $Q_\mathfrak p$ condition, 
      $n\geq \trnklg(p)$, choose $u$ such that
      $\supp(p,n-1)\subseteq u\subseteq \dom(p)$ and
      $\eta\in \POSS_{n,u}$. Then
      we can modify $p$
      by enlarging the trunk-length to $n$ and 
      replacing part of the trunk by $\eta$.
      Let us call the resulting creature $p\wedge \eta$.
      (More formally:
      $\trnk(p\wedge \eta)(m,i)=\eta(m,i)$ if $m<n$ and $i\in u$,
      and $\trnk(p)(m,i)$ otherwise.) 
    \item $p\wedge \eta\leq p$ if $\eta\in\poss(p,n)$.
    \item $\{p\wedge \eta:\, \eta\in \poss(p,n)\}$
      is predense below $p$.
    \item We set $\nugen$ to be the name for
      $\bigcup_{p\in G}\trnk(p)$. So
      $Q_{\mathfrak p}$ forces that $\nugen$ is a function with 
      domain $\omega\times J$ for some 
      $J\subseteq I^*$.
      Note that it is not guaranteed that $J=I^*$. (But
      $p$ forces that $\dom(p)\subseteq J$ and that $\nugen\restriction
      (n\times \dom(p))\in \poss(p,n)$ for all $n\in\omega$.) 
    \item If $\eta\in \poss(p,n)$, then $p\wedge \eta \forc \varphi$ iff 
        $p \forc \eta\subset \nugen \rightarrow \varphi$.
  \end{itemize}
\end{Facts}

One simple way to guarantee that  $J=I^*$
is the following: Given $i\in I$ and a creature $\cc$, we can
strengthen $\cc$ by increasing the support by (not much more than) $\{i\}$
while not decreasing the norm too much: 
\begin{Lem}\label{lem:genericdomain}
  Assume that for all $i\in I^*$
  there is an $M\in\omega$ and a $u\in [I^*]^{{<}\al0}$ containing $i$ 
  sucht that for all $n>M$ and all $\cc\in\cK(n)$ with $\nor(\cc)>M$
  there is a $\cd\in\cS(\cc)$ such that
  \begin{itemize}
    \item $\nor(\cd)>\nor(\cc)-M$ and $\norls(\cd)>\norls(\cc)-M$,
    \item $\supp(\cd)=\supp(\cc)\cup u$ and $\suppls(\cd)=\suppls(\cc)\cup u$.
  \end{itemize}
  Then the domain of 
  $\nugen$ is forced to be $\omega\times I^*$.
\end{Lem}

\begin{proof}
  Given $p\in Q_{\mathfrak p}$ and $i\in I^*$ we can
  find a $q\leq p$ such that $i\in \supp(q)$:
  For sufficiently large $n$, set $q(n)=\cd\in\cS(p(n))$ as
  above. 
\end{proof}

\subsection{Properness: Bigness and halving}

\begin{Def}
  \begin{itemize}
    \item
  For $\cc$ in $\cK(n)$ and $x>0$
  we write $\cd\in\cSp^x(\cc)$ if
  $\cd\in \cS(\cc)$,
  $\supp(\cd)=\supp(\cc)$,
  $\suppls(\cd)=\suppls(\cc)$,
  $\nor(\cd)\geq \nor(\cc)-x$ and
  $\norls(\cd)\geq \norls(\cc)-x$.
    \item
  The $n$-ml-creature $\cc$ is $(B,x)$-big, if for all 
  functions $G:\POSS_{n+1,\supp(\cc)}\to B$ 
  there is a $\cd\in \cSp^x(\cc)$ 
  and a $G':\POSS_{n,\supp(\cc)}\to B$ such that
  $G(\eta)=G'(\nu)$ for all $\eta\in \cd[\nu]$.
  I.e., modulo $\cd$ the value of 
  $G(\eta)$ only depends on $\eta\restriction n$.
    \item
  $\cK(n)$ is $(B,x)$-big, if all $\cc\in\cK(n)$
  with norm bigger than 1 are $(B,x)$-big.
  (Note that we do not require that $\cc$ has large $\norls$.)%
\footnote{%
  Of course there are some other natural definitions for bigness.  We briefly
  mention two of them, however the reader can safely skip this.
  In our setting, all these notions are more or less equivalent: 
  Firstly, we will assume that $k\coloneqq|\POSS_{n,\supp(\cc)}|$
  is ``very small'' compared to the bigness $B$. Secondly,
  $\val(\cc)$ will be determined by the sequence $(\cc[\eta])$.
  \begin{itemize}
    \item
      The $n$-ml-creature $\cc$ is weakly-$(B,x)$-big, if
      for all  $\eta\in\POSS_{n,\supp(\cc)}$ and all
      $G:\cc[\eta]\to B$ there is a $\cd\in\cSp^x(\cc)$
      such that $G\restriction \cd[\eta]$ is constant.
    \item
      The $n$-ml-creature $\cc$ is $(B,x)$-big$^*$, if
      for all 
      $G:\val(\cc)\to B$ there is a $\cd\in\cSp^x(\cc)$
      such that $G$ restricted to $\val(\cd)$ is constant.
  \end{itemize}
  We obviously get: $(B,x)$-big implies weakly-$(B,x)$-big.
  \\
  Weakly-$(B,x/k)$-big implies
  $(B,x)$-big: We just iterate 
  bigness for all $\eta\in \POSS_{n,\supp(\cc)}$,
  i.e., at most $k$ times.
  \\
  $(B^k,x)$-big$^*$
  implies $(B,x)$-big: Apply big$^*$
  to the function that maps $\mathbf f\in\val(\cc)$
  to the sequence $(\mathbf f(\eta))_{\eta\in \POSS_{n,\supp(\cc)}}$.
}
  \end{itemize}
\end{Def}

\begin{Def}
  \begin{itemize}
    \item
      A condition $p$ decides a name $\n\tau$, if there is 
      an element $x\in V$ such that $p$ forces $\n\tau=\std x$.
    \item
      $\n\tau$ is $n$-decided by $p$, if $p\wedge \eta$ decides
      $\n\tau$ for each $\eta\in\val(p,n)$. 
    \item
      $p$ essentially decides
      $\n\tau$, if $\n\tau$ is $n$-decided by $p$ for some $n$.
    \item 
      Let $r:\omega\to\omega$ be a $Q_\mathfrak p$-name.
      $p$ reads $r$ continuously, if 
      $p$ essentially decides $r(n)$ for all $n$.
    \item 
      $p$ rapidly reads $r$ (above $M$), if 
      $r\restriction n$ is $n$-decided by $p$ for all $n$ (bigger than $M$).
  \end{itemize}
\end{Def}

Sufficient bigness gets us from continuous to rapid reading:

\begin{Lem}\label{lem:fromcontinuoustorapid}
  Fic $B:\omega\to\omega$. Assume that
  \begin{itemize}
    \item
      $\cK(n)$ is $(\prod_{m<n} B(m),1)$-big for all $m\in\omega$.
    \item $p$ continuously reads $r\in \prod B$.%
\footnote{%
  I.e., $r$ is a name, $p$ forces that $r(m)<B(m)$ for all $m\in\omega$, and
  $p$ continuously ready $r$.%
}
    \item $M\geq \trnklg(p)$, and $\nor(p,m)>1$ for 
      all $m\geq M$.
  \end{itemize}
  Then there is a $q\leq p$ such that
  \begin{itemize}
    \item  $\trnklg(q)=\trnklg(p)$, 
      $\trnk(q)=\trnk(p)$, and $q(n)=p(n)$ for $\trnklg(p)\leq n<M$,
    \item $q(n)\in \cSp^1(p(n))$ for $n\geq M$,
    \item $q$ rapidly reads $r$. I.e., $r\restriction n$ is $n$-decided by $q$ 
      for all $n>M$.
  \end{itemize}
\end{Lem}
\begin{proof}
  For $n\in\omega$, let $h(n)\geq 0$ be maximal
  such that $r\restriction h(n)$ is $n$-decided by $p$.
  So $h(n)$ is a weakly increasing, unbounded function.
  Set 
  \[
    x_{n,l}=r\restriction \min(h(n),l).
  \]
  Note that $x_{n,n}$ is $n$-determined by $p$,
  and that there are
  at most $\prod_{m<l} B(m)$ many possibilities for $x_{n,l}$.

  For all $n\geq M$, we define by downward induction for
  $l=n,n-1,\dots,M+1,M$ the creatures
  $\cd_{n,l}\in\cSp^1(p(l))$ and
  the function $\psi_{n,l}$ with domain $\poss(p,n)$: 
  \begin{itemize}
    \item
      $\cd_{n,n}=p(n)$,
      $\psi_{n,n}(\eta)$ is the value of $x_{n,n}$ as forced 
      by $p\wedge \eta$.
    \item
      For $l<n$ and $\eta\in\poss(p,l+1)$ we know by induction that
      $\psi_{n,l+1}(\eta)$ is a potential value for $x_{n,l+1}$.
      Let $\psi^-_{n,l+1}(\eta)$ be the corresponding value of 
      $x_{n,l}$. Using bigness, we get a
      $\cd_{n,l}\in\cSp^1(p(l))$ such that 
      $\psi^-_{n,l+1}(\eta)$ only depends on
      $\eta\restriction l\in \poss(p,l)$.
      We set  $\psi_{n,l}(\eta\restriction l)$ to be this value
      $\psi^-_{n,l+1}(\eta)$.
  \end{itemize}
  For every $n\in \omega$, set
  $y_n=(\val(\cd_{n,l}),\psi_{n,l})_{M\leq l \leq n}$.
  For all $l$ there are only finitely many values
  for $\val(\cd_{n,l})$  and for $\psi_{n,l}$.
  So the set of the sequences $y_n$ together with their
  initial sequences form a 
  finite splitting tree. Using K\"onig's Lemma, we get an infinite
  branch: A  sequence $(\cd_l^*,\psi_l^*)_{l\geq M}$
  such that $\cd_l^*\in\cSp^1(p(l))$ 
  and such that for all $n$ the sequence
  $y^*_n=(\val(\cd_l^*),\psi_l^*)_{M\leq l<n}$
  is initial sequence of $y_{m}$ for some $m>n$.

  We define $q\leq p$ by $q(l)=p(l)$ for $n<M$ and $q(l)=\cd^*_{l}$ otherwise
  (and, of course, $\trnk(q)=\trnk(p)$).

  Fix $n>M$.
  We claim that $r\restriction n$
  is $n$-decided  by $q$.

  Pick some $m$ such that $h(m)>n$
  and some $k$ such that $y^*_m$
  is initial sequence of $y_k$.
  Recall the inductive construction of $\cd_{k,l}$: 
  \begin{equation}
    \parbox{0.8\columnwidth}{Modulo $p$ and
      $\cd_{k,n},\cd_{k,n-1},\dots,\cd_{k,k}$ any $\eta\in\poss(p,n)$ 
     already decides $x_{k,n}$.}
  \end{equation}
  Also, $x_{k,n}$ contains $r\restriction n$ (since $h(k)>n$). In fact even
  $h(m)>n$, so $r\restriction n$ is decided by $p\wedge \nu$
  for all $\nu\in\poss(p,m)$. Therefore we can improve the previous equation:
  \begin{equation}
    \parbox{0.8\columnwidth}{Modulo $p$ and
      $\cd_{k,m-1},\dots,\cd_{k,k}$ any $\eta\in\poss(p,n)$ 
     already decides $x_{k,n}$.}
  \end{equation}
  Now recall that $\cd_{k,m-1},\dots,\cd_{k,k}$ are conditions in $q$,
  so $x_{k,n}$ (and therefore $r\restriction n$) is $n$-decided by $q$.
\end{proof}

To get properness, we need another well established creature forcing concept:
\begin{Def}
  The $n$-ml-creature $\cc$  is $x$-halving, if there 
  is a $\half(\cc)\in\cSp^x(\cc)$ satisfying the following:
  If $\cd\in\cS(\half(\cc))$ has non-zero norm,
  then there is a $\cd'$ (called
  the un-halved version of $\cd$) satisfying:
  \begin{itemize}
    \item $\cd'\in\cS(\cc)$,
    \item $\supp(\cd')=\supp(\cd)$, and $\suppls(\cd')=\suppls(\cd)$,
    \item $\nor(\cd')\geq \nor(\cc)-x$ and
      $\norls(\cd')\geq \norls(\cc)-x$,
    \item $\cd'[\eta]\subseteq \cd[\eta]$ for
      all $\eta\in \POSS_{n,\supp(\cd)}$.%
\footnote{%
  An alternative, stronger definition would be: 
  $\val(\cd')\subseteq \val(\cd)$.
  In the special case mentioned in Remark~\ref{rem:simplecase}
  these versions are equivalent.
}
  \end{itemize}
      $\cK(n)$ is $x$-halving, if all $\cc\in \cK(n)$
      with $\nor(\cc)>1$ are $x$-halving.
      (Note that we do not require $\norls(\cc)>1$.)
\end{Def}

\begin{Def}\label{def:sufficient}
  A forcing parameter $\mathfrak p$ has
  sufficient bigness and halving,
  if there is an increasing function $\maxposs:\omega\to\omega$ such that
  for all $n\in\omega$
  \begin{enumerate}
    \item $|\poss(p,n)|<\maxposs(n)$
      for all $p\in Q_\mathfrak p$.
    \item $ \cK(n)$ is $(2,1)$-big.
    \item $\cK(n)$ is $1/\maxposs(n)$-halving.
  \end{enumerate}
\end{Def}

\begin{Rem}\label{rem:sufficient}
  The natural way to  guarantee~(1) is  the following:
  There is an increasing
  function $\maxsupp:\omega\to\omega$ such that
  for every $n\in\omega$
  \begin{itemize}
    \item every $n$-ml-creature $\cc$ satisfies
      $|\supp(\cc)|<\maxsupp(n)$,
    \item There is an $M(n)\in\omega$ such that
      $|\POSS^*_{=m,\{i\}}|<M(n)$ for all $i\in I^*$ and $m<n$, and
    \item $\maxposs(n)\geq M(n)^{(n* \maxsupp(n-1))}$.
  \end{itemize}
  A bit of care will be required to construct such creatures, since
  on the other hand we will also need
  \begin{itemize}
    \item the norm of a creature does not decrease by, say, more than $1$ if we
      ``make the support twice as big'' (we need this 
        to prove $\al2$-cc, cf.\ Definition~\ref{def:localdelta}), and
    \item there is an $n$-ml-creature $c$ with $\nor(\cc)\geq n$ (this
     guarantees that $Q_\mathfrak p$ is nonempty).
  \end{itemize}
\end{Rem}

\begin{Lem}\label{lem:essentialdecide}
  Assume that $\mathfrak p$ has sufficient bigness and halving,
  that $\n\tau$ is the name for an element of $V$, 
  that $p_0\in Q_\mathfrak p$, that $M_0\geq \trnklg(p_0)$, $n_0\geq 1$ and
  $\nor(p_0,m)\geq n_0+2$ for all $m\geq M_0$. Then there is a 
  $q\leq p_0$ such that\footnote{note that as opposed to
  the previous lemma, the supports of $q(n)$ will generally be bigger than
  those of $p(n)$.}
  \begin{itemize}
    \item $q$ essentially decides $\n\tau$,
    \item $q(m)=p_0(m)$ for $\trnklg(p_0)\leq m<M_0$,
    \item $\nor (q,m)\geq n_0$ for all $m\geq M_0$.
  \end{itemize}
\end{Lem}
Then the usual standard argument gives us properness and
$\omega^\omega$-bounding, and Lemma~\ref{lem:fromcontinuoustorapid} gives us
rapid reading:
\begin{Cor}\label{cor:proper}
  Assume that $\mathfrak p$ has sufficient bigness and halfing.
  \begin{itemize}
    \item $Q_{\mathfrak p}$ is proper and $\omega^\omega$-bounding.
    \item If additionally every $\cK(n)$ is $(\prod_{m<n} B(m),1)$-big,
      we get rapid reading: If
      $r$ is a name for an element of $\prod B$ 
      then for every $p$ there is a $q\leq p$ such that
      $r\restriction m $ is $m$-decided by $q$ for all $m\in\omega$. 
  \end{itemize}
\end{Cor}
Let us first give a sketch of the (standard) argument of the Corollary:
\begin{proof}
  \begin{itemize}
    \item $\omega^\omega$-bounding: Assume that $f$ is a name
      for a function from $\omega$ to $\omega$ and that $p_0$ is in
      $Q_\mathfrak p$. Using the previous lemma, we iteratively
      construct $p_{n+1}\leq p_n$ and $h_n$ such that 
      \begin{itemize}
        \item  $p_{n+1}$ essentially decides  $f(n)$,
        \item  $p_{n+1}(m)=p_n(m)$ for all $m<h_n$,
        \item  for some $i\in\dom(p,n)$ (picked by suitable bookkeeping)
           there is an $m<h$ such that $i\in\suppls(p_n,m)$ and
           $\norls(p_n,m)>n$,
        \item  $\nor(p_{n+1},m)>n$ for all $m\geq h_n$.
      \end{itemize}
      This guarantees that the sequence of the $p_n$'s has a limit $q$,
      which essentially decides all $f(n)$. This in turn implies that
      (modulo $q$) there are only finitely many possibilities for each 
      $f(n)$, which gives us $\omega^\omega$-bounding.
   \item Properness: Fix $N\esm H(\chi)$ and $p_0\in N$.
      We need a $q\leq p$ which is $N$-generic, i.e., which forces
      that $\n\tau[G]\in N$ for all names for ordinals that are in $N$.
      Enumerate all these names as $\{\n\tau_0,\n\tau_1\dots\}$.
      Now do the same as above, but instead of $f(n)$
      use $\tau_n$; and construct each $p_n$ inside of $N$. (The whole
      sequence of the $p_n$'s cannot be in $N$, of course.)
      Then $q$ leaves only finitely many possibilities for each $\n\tau_n$,
      each possibility being element of $N$, which gives properness.
  \end{itemize}
\end{proof}

\begin{proof}[Proof of Lemma~\ref{lem:essentialdecide}]
  {\bf (a) Halving, the single step $S^{\rm e}(p,M,n)$:}

  Assume that
  \begin{itemize}
    \item $p\in P$, 
    \item $M\geq \trnklg(p)$,
    \item $n\geq 1$, $\nor(p,m)>n$ for all $m\geq M$.
  \end{itemize}
  We now define  $S^{\rm e}(p,M,n)\leq p$.
  Enumerate $\poss(p,M)$ as $\eta^1,\dots,\eta^{l}$.
  So $l \leq \maxposs(M)$.
  Set $p^{0}=p$. For $1\leq k\leq l$, pick $p^k$ such that
  \begin{itemize}
    \item $\trnklg(p^k)=M$ and  $p^k\leq p^{k-1}\wedge \eta_k$.
      (So in particular, $\trnk(p^k)\restriction \dom(p)=\eta_k$.)
    \item For all $m\geq M$, $\nor(p^k,m)>n-k/\maxposs(M)$.
    \item One of the following cases holds:
    \begin{description}
      \item[dec]  $p^{k}$ essentially decides $\n\tau$, or
      \item[half] it is not possible to satisfy case {\em dec}, then
         $p^{k}(m)=\half(p^{k-1}(m))$ for all $m>M$.
    \end{description}
  \end{itemize}
  So in  case {\bf half}, we get $\dom(p^k)=\dom(p^{k-1})$, but in
  case {\bf dec} the domain will generally increase.

  We now define $q=S^{\rm e}(p,M,n)$ by $q(m)=p(m)$ for $m<M$ and
  $q(m)=p^{l}(m)$ otherwise.\footnote{And, of course, we set
    $\trnk(q,i)=\trnk(p,i)$ if $i\in\dom(p)$ and $\trnk(q,i)=\trnk(p^l,i)$
    otherwise.}
  Note that $\nor(q,m)>n-1$ for all $m\geq M$.

  {\bf (b) Iterating the single step:}

  Given $p_0$, $M_0$ and $n_0$ as in the Lemma, we 
  inducitvely construct $p_k$ and  $M_k$ for $k\geq 1$:
  \begin{itemize}
    \item Choose by some bookkeeping an $\alpha\in\dom(p_{k-1})$.
    \item Choose 
\begin{equation}\label{eq:e5iojh23}
  M_k>k+M_0
\end{equation}
      big enough such that
      \begin{itemize}
        \item there is an $l<M_k$
          with $\alpha\in\suppls(p_{k-1},l)$ and $\norls(p_{k-1},l)>k$,
        \item $\nor(p_{k-1}(m))>k+n_0+2$ for all $m>M_k$.
      \end{itemize}
    \item Let $p_k$ be $S^{\rm e}(p_{k-1},M_k,k+n_0+2)$.
  \end{itemize}
  Assuming adequate bookkeeping, the sequence $p_k$ has a limit 
  $q_0\leq p_0$, and
  $\nor(q_0,m)>n_0+1$ for all $m\geq M_0$.
 
  {\bf (c) Bigness, thinning out $q_0$}

  We now thin out $q_0$, using bigness 
  in a way similar to the proof
  of Lemma~\ref{lem:fromcontinuoustorapid}.

  For all $n\in \omega$, we define by downward induction for
  $l=n,n-1,\dots,M_0+1,M_0$, a subset
  $\Lambda_{n,l}$ of $\poss(q_0,l)$ and ml-creatures 
  $\cd_{n,l}\in\cSp^1(q_0(l))$:
  \begin{itemize}
    \item $\cd_{n,n}=q_0(l)$; and 
      $\eta\in\Lambda_{n,n}$ iff $q_0\wedge \eta$ essentially decides $\n\tau$. 
    \item For $l<n$, we use bigness to get
      $\cd_{n,l}\in\cSp^1(q_0(l))$ such that 
      for all $\eta\in\poss(q_0,l)$ either
      $\cd_{n,l}[\eta]\subseteq \Lambda_{n,l+1}$
      or $\cd_{n,l}[\eta]\cap \Lambda_{n,l+1}=0$.
      We set
      $\Lambda_{n,l}$ to be the set of those 
      $\eta\in\poss(q_0,l)$ such that 
      $\cd_{n,l}[\eta]\subseteq \Lambda_{n,l+1}$.
  \end{itemize}
  So by this construction we get:
  If $\eta\in \poss(q_0,M_0)\cap \Lambda_{n,M_0}$
  then every $\nu\in\poss(q_0,n)$ that extends $\eta$ and is compatible with
  $(\cd_{n,l})_{M_0\leq l<n}$ satisfies $q_0\wedge \nu$ essentially 
  decides $\n\tau$.

  If on the other hand
  \begin{itemize}
    \item $\eta\in \poss(q_0,M_0)\setminus \Lambda_{n,M_0}$,
    \item $\nu$ is in $\poss(q_0,M)$ for some $M_0\leq M\leq n$,
    \item $\nu$ extends $\eta$, and
    \item $\nu$ is compatible with $(\cd_{n,l})_{M_0\leq l<M}$, then
  \end{itemize}
  \begin{equation}\label{eq:gugki}
    q_0\wedge \nu\text{ does not essentially decide }\n\tau.
  \end{equation}

  We claim that there is some $n_0\geq M_0$ such  that
  \begin{equation}\label{eq:ojho235}
    \poss(q_0,M_0)\subseteq \Lambda_{n_0,M_0}.
  \end{equation}
  Then we define $q\leq q_0$ by $q(m)=\cd_{n_0,m}$
  for $M_0\leq m\leq n_0$ and $q(m)=q_0(m)$ for $m>n_0$.
  According to the definition of
  $\Lambda_{n_0,M_0}$, we know that $q_0\wedge \nu$ essentially 
  decides $\n\tau$ for all $\nu\in\poss(q,n_0)$,
  so $q$ essentially decides $\n\tau$. This finishes
  the proof of the Lemma, since 
  $q$ satisfies the other requirements as well.
  
  So it remains to show~\eqref{eq:ojho235}.
  For every $n\in \omega$, we define the finite sequence 
  \[
    x_n=(\val(\cd_{n,l}),\Lambda_{n,l})_{M_0\leq l \leq n}.
  \]
  For each $l$, there are only finitely many possibilities
  for $\val(\cd_{n,l})$ and for $\Lambda_{n,l}$,
  so the set of the sequences $x_n$ together with their
  initial sequences form a 
  finite splitting tree. Using K\"onig's Lemma, we get an
  infinite branch. So we get a sequence
  $(\cd^*_l,\Lambda^*_l)_{M_0\leq l \leq \omega}$ such that 
  $\cd^*_l\in \cSp^1(q_0(l))$ and
  for all $n$ there is an $m>n$ such that the sequence
  \[
    x^*_n=(\val(\cd^*_{l}),\Lambda^*_{l})_{M_0\leq l \leq n}
  \]
  is an inital sequence of $x_m$.

  We claim
  \begin{equation}\label{eq:huiwt}
    \poss(q_0,M_0)\subseteq \Lambda^*_{M_0}.
  \end{equation}
  Then we get~\eqref{eq:ojho235} by picking any $n_0$ such that 
  $\Lambda_{n_0,M_0}=\Lambda^*_{M_0}$.

  To show~\eqref{eq:huiwt},
  assume towards a contradiction that there is some 
  $\eta_0\in\poss(q_0,M_0)\setminus \Lambda^*_{M_0}$.
  Define $q_1\leq q_0$ by 
  $q_1(l)=q_0(l)$ if $l<M_0$ and 
  $q_1(l)=\cd^*_{l}$ otherwise.
  Find an $s\leq q_1\wedge \eta_0$ deciding $\n\tau$.
  Without loss of generality, $\trnklg(s)=M_k>M_0$ 
  for some $k$, where $M_k$ was chosen in \eqref{eq:e5iojh23}.  
  Also we can assume $\nor(s,m)>2$ for all $m>\trnklg(s)$.
  Let $\trnk(s)$ extend some 
  $\nu\in\poss(q_1,M_k)\subseteq \poss(q_0,M_k)$.
  In particular, $\nu$ extends $\eta_0$.
  We claim:
  \begin{equation}\label{eq:o3i5}
    \text{$q_0\wedge \nu$ does not essentially decide $\n\tau$}
  \end{equation}
  Pick $m$ such that $x_m$ extends $x^*_{M_k}$.
  In particular, $\Lambda_{m,M_0}=\Lambda^*_{M_0}$, so 
  $\eta_0\notin \Lambda_{m,M_0}$.
  Since $\nu\in \poss(q_1,M_k)$,
  $\nu$ is compatible with the sequence $\val(\cd^*_l)_{M_0\leq l<M_k}$
  and $\val(\cd^*_l)=\val(\cd_{m,l})$. 
  So by~\eqref{eq:gugki} we get that
  $q_0\wedge \nu$ does not essentially decide $\n\tau$.
  This proves~\eqref{eq:o3i5}.

  By~\eqref{eq:o3i5} we know: when we were dealing with $\nu$ in stage $k$, 
  we were in the {\bf half}-case. In particular,
  $s$ is stronger than some $p_{k-1}^l$ that resulted
  from halving $p_{k-1}^{l-1}$.
  Let $M'$ be such that $\nor(s,m)>k+n_0+2$ 
  for all $m\geq M'$. We can now un-halve 
  $s(m)$ for all $h_k\leq m < M'$ (and leave
  it unchanged above $M'$), resulting in 
  a condition $s'$ that is stronger than $p_{k-1}^{l-1}$ and
  essentially decides $\n\tau$,
  a contradiction to the fact that $p_{k-1}^l$ was
  constructed using the {\bf half}-case.
  So we have shown~\eqref{eq:huiwt}.
\end{proof}

\begin{Rem}
  The proof actually shows that it is not required that all
  $n$-ml-creatures are $1/\maxposs(n)$-halving. It is enough to
  have an infinite set $w\subseteq \omega$ such that 
  for all $M\in w$ and $n\geq M$ every $n$-ml-creature is 
  $1/\maxposs(M)$-halving. (Just choose all the $M_k$ in the proof
  to be in $w$.)
\end{Rem}

\subsection{$\al2$-cc}

To preserve all cofinalities, we will use $\al2$-cc in addition to properness.
To guarantee that $Q_{\mathfrak p}$ is $\al2$-cc, we need additional properties
of $\mathfrak p$ and we have to assume CH in the ground model.

We will argue as follows: Assume towards a contradiction that $A$ is an
antichain of size $\al2$.  By a standard $\Delta$-system argument we can assume
that any two conditions in $A$ have (more or less) disjoint domain; we assume
that there are only continuum many different conditions ``modulo isomorphism of
the domain''; and then we have to argue that two identical (modulo domain)
conditions with disjoint domain are compatible.

There are many ways to achive this, one sufficient conditions is the
following:
\begin{Def}\label{def:localdelta}
  Fix $n\in\omega$.
  The $n$-crature-parameter $\mathfrak p(n)$ has the local
  $\Delta$-property, if we can assign one of continuum
  many\footnote{In practise, we can get finitely many.} ``local types''
  to each pair $(\cc,\bar i)$,
  where $\cc$ is an $n$-ml-creatue and
  $\bar i: |\supp(\cc)|\to \supp(\cc)$ is bijective,
  such that the following holds:
  \\
  {\bf If}
  \begin{itemize}
    \item $(\cc_1,\bar i_1)$ and $(\cc_2,\bar i_2)$ are as above
      and have the same local type,
    \item $\nor(\cc_1)=\nor(\cc_2)>1$ and $\norls(\cc_1)=\norls(\cc_2)$,
    \item the enumerations $\bar i_1$ and $\bar i_2$ agree
      on $\supp(\cc_1)\cap \supp(\cc_2)$.
      \\
      More formally: if $i\in \supp(\cc_1)\cap \supp(\cc_2)$, then
      there is an $m$ such that $\bar i_1(m)=\bar i_2(m)=i$,
  \end{itemize}
  {\bf then} there is a $\cd\in \cS(\cc_1)\cap \cS(\cc_2)$ such  that
  \begin{itemize}
    \item  $\supp(\cd)=\supp(\cc_1)\cup \supp(\cc_2)$ and
      $\suppls(\cd)=\suppls(\cc_1)\cup \suppls(\cc_2)$,
    \item $\nor(\cd)\geq \nor(\cc_1)-1$ and
      $\norls(\cd)\geq \norls(\cc_1)-1$.
  \end{itemize}
\end{Def}

\begin{Lem}\label{lem:al2cc}
  Assume CH and that $\mathfrak p(n)$ has the local
  $\Delta$-property for all $n$.
  Then $Q_\mathfrak p$ is $\al2$-cc.
\end{Lem}

\begin{proof}
Assume towards a contradiction that $A$ is an antichain of size $\al2$. 
We can assume that there is a $\Delta\subseteq I^*$ such that
$\dom(p)\cap \dom(q)=\Delta$ for all $p\neq q$ in $A$,
and that $|\supp(p)|=M\leq \omega$ for all $p\in A$.
Pick for all $p\in A$ a bijection
$\bar i^p: M\to \dom(p)$.

We can also assume that the following objects and statements
do not depend on the choice of $p\in A$ for
$i^\Delta\in \Delta,m<M$ and $n\in\omega$:
\begin{itemize}
  \item The trunk of $p$ ``modulo the enumeration of the domain'', i.e.,
    $\trnklg(p)$, $\trnklg(p,\bar i^p(m))$ and $\trnk(p,\bar i^p(m))$.
  \item The norms, $\nor(p,n)$, $\norls(p,n)$.
  \item 
    The local type of $(p(n),\bar j^p_n)$, where $\bar j^p_n$ is
    $\bar i^p$ restricted to $\supp(p,n)$.%
\footnote{%
  More formally,
  $\bar j^p_n: |\supp(p,n)|\to \supp(p,n)$
  is defined by $\bar j^p_n(l)=\bar i^p(k)$ for
  the minimal $k$ such that $\bar i^p(k)\in \supp(p,n)\setminus
  \bar {j^p_n}'' l$.%
}
  \item Whether $\bar i^p(m)\in \supp(p,n)$.
  \item Whether $\bar i^p(m)=i^\Delta$.
\end{itemize}

Now pick  $p\neq q$ in $A$. We show towards a contradiction that $p$ and $q$ are
compatible: Pick $h$ such that $\nor(p,n)>1$ for all $n\geq
h$. The local types of $(p(n),\bar j^p_n)$ and $(q(n),\bar j^q_n)$
are the same. If $i^\Delta\in \supp(p,n)\cap \supp(q,n)$, then
$i^\Delta=\bar i^p(m)=\bar i^q(m)$ for some $m<M$, 
and $\bar i^p(k)\in \supp(p,n)$ iff $\bar i^q(k)\in \supp(q,n)$ for 
all $k\leq m$,
therefore $i^\Delta=\bar j^p_n(l)=\bar j^q_n(l)$ for some $l$.
So we can apply the local $\Delta$ property and 
get $\cd\in\cS(p(n))\cap \cS(q(n))$.
The sequence of these creatures, together with the union of the stems of $p$
and $q$, from a condition $r\leq p,q$.
\end{proof}

\section{Continuum many invariants}
We now apply this creature forcing construction (actually, only the pure
lim-sup case and the simplified setting described in
Remark~\ref{rem:simplecase}) to improve the result of {\em Decisive
Creatures}~\cite{MR2499421}.  We have to make sure to define the ml-creatures
and the norms in a way to satisfy sufficient bigness and halfing (see
Definition~\ref{def:sufficient} and the Remark following it).  Once we have
done this, it turns out that the rest of the proof of the Main Theorem is a
rather straightforward modification of the proof in~\cite{MR2499421}.

\subsection{Atomic creatures, decisiveness}

We will build the ml-creatures from simpler creatures, which we call
atomic creatures.
An atomic parameter is a tuple $a=(A,\cK,\val,\nor,\cS)$ such that
\begin{itemize}
  \item $A$ is a finite set.
  \item $\cK$ is a finite set (the set of $a$-atomic creatures),
  \item $\val$, $\nor$ and $\cS$ are functions with domain $\cK$
\end{itemize}
such that for all $a$-atomic creatures $w\in\cK$ the following holds:
\begin{itemize}
  \item $\nor(w)\geq 0$,
  \item $\val(w)\subseteq A$ is nonempty,
  \item $\cS(w)$ is a subset of $\cK$,
  \item $w\in\cS(w)$; and if $w_2\in\cS(w_1)$ and $w_3\in\cS(w_2)$
        then $w_3\in\cS(w_1)$,
  \item if $v\in\cS(w)$ then $\val(v)\subseteq \val(w)$
    and $\nor(v)\leq \nor(w)$,
  \item if $|\val(w)|=1$ then $\nor(w)<1$.
\end{itemize}

As usual we get notions of bigness and halving, as well as decisiveness
as introduced in~\cite{MR2499421}:

\begin{itemize}
  \item $v\in \cSp^x(w)$ means $v\in\cS(w)$ and 
    $\nor(v)>\nor(w)-x$.
  \item $w\in\cK$ is $(B,x)$-big, if for all
    $F:\val(w)\to B$ there is a $v\in\cSp^x(w)$
    such that $F\restriction \val(v)$ is constant.
  \item $w$ is hereditary $(B,x)$-big, if every $v\in\cS(w)$
    with norm at least 1 is $(B,x)$-big.
  \item The atomic parameter $a$
    is $(B,x)$-big, if every $w\in\cK$
    with norm at least 1 is $(B,x)$-big.
  \item $w\in\cK$ is $x$-halving, if there is a 
    $\half(w)\in\cSp^x(w)$ such that for all $v\in\cS(\half(w))$
    with norm bigger than $0$
    there is
    a $v'\in \cSp^x(w)$ with $\val(v')\subseteq \val(v)$.
    We call this $v'$
    ``unhalved version of $v$'', or we say that we ``unhalve $v$'' to
    get $v'$.
  \item The atomic parameter $a$
    is $x$-halving, if every $w\in\cK$ with norm bigger than
    1 is $x$-halving.
  \item $w\in\cK$ is $(K,m,x)$-decisive, if
    there are $v^-,v^+\in \cSp^x(w)$ such that 
\begin{equation}
  |\val(v^-)|\leq K
  \quad \text{ and } \quad
  v^+\text{ is hereditarily }(2^{K^m},x)\text{-big}.
\end{equation}
    $v^-$ is called a $K${\em -small successor}, and
    $v^+$ a $K${\em-big successor} of $w$.
  \item $w$ is $(m,x)$-decisive if $w$ is $(K,m,x)$-decisive for some $K$.
  \item $\cK$ is $(m,x)$-decisive if every $w\in \cK$ with
    $\nor(w)>1$ is $(m,x)$-decisive.
  \item 
  An atomic-parameter is $M$-nice with maximal norm $m$, if
  it is $(2^M,1/M^2)$-big, $1/M$-halving and $(M,1/M^2)$-decisive
  and $m=\max(\nor(w):\, w\in\cK)$.
\end{itemize}

\begin{Facts}\label{facts:214f}
  \begin{enumerate}
    \item\label{item:existence}
      Given $M,m\in\omega$ there is an $M$-nice atomic-parameter
      with maximal norm $m$.
    \item\label{item:atomlem1}
      Assume that an atomic paramter is $M$-nice,
      that $\nor(w_i)>2$ for all $i\in M$, 
      and that $F:\prod_{i\in M}\val(w_i)\to 2^{M}$.
      Then there are $v_i\in\cSp^{1/M}(w_i)$ such that
      $F\restriction \prod_{i\in m}\val(v_i)$ is constant.
  \end{enumerate}
\end{Facts}

\begin{proof}
  This is shown in ~\cite{MR2499421}: (1) is Lemma~6.1, (2) is
  Corollary~4.4.
\end{proof}

\subsection{The forcing}

\begin{Def}\label{def:initialsequence}
  We define by induction on $n\in\omega$ the natural numbers
  $\maxposs(n)$,
  $\maxnor(n)$, 
  $\maxsupp(n)$, 
  $\Bmin(n)$,
  $k^*(n)$, 
  $\gmin(n)$ and $\fmax(n)$; as well as
  $f_{n,m}$ and $g_{n,m}$ for  $m\in k^*(n)$:
\begin{enumerate}
  \item Set $\fmax(-1)=\maxsupp(-1)=1$.
  \item\label{item:maxposs} Set $\maxposs(n)=1+(\fmax(n-1))^{n \maxsupp(n-1)}$.
    \\
    (By induction, we will see that 
    $|\poss(p,n)|<\maxposs(n)$ for every condition $p$.)
  \item Set $\maxnor(n)=1+2^{n\cdot \maxposs(n)}$.
    \\
    (This will later be used to guarantee there is an
    $n$-ml-creature with norm $n$, i.e., that $Q_\mathfrak p$ is nonempty.) 
  \item\label{item:maxsupp} Set $\maxsupp(n)=1+2^{\maxnor(n)}$.
    \\
    (We will later define the $n$-ml-creatures so that
    $|\supp(\cc)|\leq \maxsupp(n)$ for all $\cc\in\cK(n)$.)
  \item\label{item:bmin} Pick $\Bmin(n)$ large with respect to $\maxsupp(n)$.
     \\
     More specifically: larger than $\fmax(n-1)^{n\fmax(n-1)^{1+(n\maxsupp(n))}}$ and larger than $2\maxsupp(n)^2$.
  \item Pick $k^*(n)$ large with respect to $\Bmin(n)$, which means that
    we can fix a $\Bmin(n)$-nice atomic paramter
    $a_{n,*}=(k^*_{n}, \cK_{n,*}, \val_{n,*}, \nor_{n,*}, \cS_{n,*})$ with
    maximal norm $\maxnor(n)$. 
  \item\label{item:gmin} Pick $\gmin(n)=g_{n,0}$ 
     large with respect to $k^*(n)$. 
     \\
     More specifically, we will need: larger than 
     $\fmax(n-1)^{n\maxsupp(n)}\cdot \maxposs(n)\cdot k^*(n)^{\maxsupp(n)}$
     and than $\fmax(n-1)^{n\fmax(n-1)}$.
  \item Pick $f_{n,m}$ large with respect to $g_{n,m}$,
    which means that we can fix an $g_{n,m}$-nice atomic parameter
    $a_{n,m}=(f_{n,m}, \cK_{n,m}, \val_{n,m}, \nor_{n,m}, \cS_{n,m})$ with
    maximal norm $\maxnor(n)$.
  \item\label{item:g} Pick $g_{n,m+1}$ large with respect to $f_{n,m}$.
     \\
     More specifically, we need: larger than $(f_{n,m})^{{f_{n,m}}^{k^*(n)}}$.
  \item Set $\fmax(n)=f_{n,k^*(n)-1}$.
\end{enumerate}
\end{Def}

We choose an index set $I^*$ containing $\mu$ and 
sets $I_\epsilon$ for all $\epsilon\in \mu$:
\begin{itemize}
  \item
    For every $\epsilon$ in $\mu$, pick
    some $I_\epsilon$ of size $\kappa_\epsilon$
    such that $\mu$ and all the $I_\epsilon$ are
    pairwise disjoint. Set $I^*=\mu\cup\bigcup_{\epsilon\in\mu}I_\epsilon$.
  \item
    We define $\varepsilon:I^*\setminus\mu \to I^*$ by
    $\varepsilon(\alpha)=\epsilon$ for
    $\alpha\in I_{\epsilon}$.
    A subset $u$ of $I^*$ is $\varepsilon$-closed,
    if for all $\varepsilon(\alpha)\in u$
    for all $\alpha\in u\setminus\mu$.
\end{itemize}

For $\epsilon\in \mu$ we set $\POSS_{=m,\{\epsilon\}}$ to be $k^*(m)$,
and for $\alpha\in I^*\setminus \mu$ we set $\POSS_{=m,\{\alpha\}}$
to be $\fmax(m)$.

\begin{Def}
  We define the ml-parameter $\mathfrak p(n)$:
  An $n$-ml-creature $\cc$ is a triple $(u^\cc,\bar w^\cc,d^\cc)$
  satisfying the following:
  \begin{itemize}
    \item $u^\cc\subset I^*$ is nonempty, $\epsilon$-closed,
     and of size at most $\maxsupp(n)$.
    \item $\bar w^\cc$ consists of the sequences
      $(w^\cc_\epsilon)_{\epsilon\in u^\cc\cap \mu}$
      and 
      $(w^\cc_{\alpha,k})_{\alpha\in u^\cc\cap I_\epsilon, k\in \val(w^\cc_\epsilon})$
      such that
      $w^\cc_\epsilon$ is an $a_{n,*}$-creature
      and $w^\cc_{\alpha,k}$ is an $a_{n,k}$-creature.
      We will write $A^\cc_\epsilon$ (or $A^\cc_{\alpha,k}$)
      for $\val(w^\cc_\epsilon)$ (or $\val(w^\cc_{\alpha,k})$,
      respectively).
    \item $d^\cc\in \mathbb R_{\geq 0}$.\footnote{We could restrict this to
      a countable set; moreover given $\bar w^\cc$
      we can even restrict $d^\cc$ to a finite set.}
  \end{itemize}
  Given such an nl-creature $\cc$, we
  define the creature-properties of $\cc$ as follows:
  \begin{itemize}
    \item $\supp(\cc)\coloneqq u^\cc$.
    \item $\val(\cc)$ is the set of those $\mathbf f\in \VAL_{n,u^\cc}$
      that satisfy the following for all $\eta\in\POSS_{n,u^\cc}$:
      If $\epsilon\in u^\cc\cap \mu$,
      then $\mathbf f(\eta)(n,\epsilon)\in A^\cc_\epsilon$,
      and if $\alpha\in u^\cc\cap I_\epsilon$ and
      $\mathbf f(\eta)(n,\epsilon)=k$ then
      $\mathbf f(\eta)(n,\alpha)\in A^\cc_{\alpha,k}$.
    \item $\nor(\cc)\coloneqq 
      (1/\maxposs(n))\,\cdot\,
      \log_2\left[\minnor(\cc) -\log_2(|\supp(\cc)|)-d \right]$,
      where we set $\minnor$ to be the minimum of the norms of all
      atomic creatures used, i.e.,
\begin{equation}
  \minnor(\cc)\coloneqq 
      \min\left(\{\nor_{n,*}(w^\cc_\epsilon):\, \epsilon\in u\cap \mu\}\cup
        \{\nor_{n,k}(w^\cc_{\alpha,k}):\, \alpha\in u\cap I_\epsilon,k\in A^\cc_\epsilon\}\right).
\end{equation}
      (If $\nor(c)$ would be negative or undefined when calculated this way, we
      set it $0$.) 
    \item $\suppls(\cc)\coloneqq \supp(\cc)$ and $\norls(\cc)\coloneqq n$ 
      (so here we have the pure lim-inf case).
  \end{itemize}
\end{Def}

So our ml-creatures have rather ``restricted memory'', they
only do not ``look down'' at all, and horizontally
only ``look from $\alpha$ to $\epsilon(\alpha)$''. More exactly:
\begin{Fact}
  $\eta\in \poss(p,n)$ iff
  \begin{itemize}
    \item $\eta$ is compatible with $\trnk(p)$,
    \item for all $m$ with $\trnklg(p)\leq m< n$, $\cc:=p(m)$,
      and $\alpha \in I_\epsilon\cap \supp(\cc)$
      we have: $\eta(m,\epsilon)\in A^\cc_\epsilon$ and
      $\eta(m,\alpha)\in A^\cc_{\alpha,\eta(m,\epsilon}$.
  \end{itemize}
\end{Fact}

\begin{Lem}\label{lem:qwh232}
    \begin{itemize}
      \item $\cK(n)$ is $(\fmax(n-1)^{n \fmax(n-1)},1)$-big.
      \item $\cK(n)$ is $1/\maxposs(n)$-halving.
      \item $\mathfrak p$ satisfies the local $\Delta$-property.
      \item The generic element lives on all of $I^*$ (i.e., the 
        domain of the generic sequence is $\omega\times I^*$).
    \end{itemize}
\end{Lem}
So we can use Lemma~\ref{lem:al2cc} and Corollary~\ref{cor:proper}
(since $\maxposs(n)$ witnesses that $\mathfrak p$ has 
sufficient bigness and halving, as defined in~\ref{def:sufficient}),
and get:
\begin{Cor}\label{cor:gurk} 
  $Q_\mathfrak p$ is proper, $\omega^\omega$-bounding and $\al2$-cc.
  If $p\in Q_\mathfrak p$
  forces that $r(n) < \fmax(n)^{\fmax(n)}$ for all $n$, then there
  is a $q\leq p$ that $n$-decides $r\restriction n$ for all $n$.
\end{Cor}

\begin{proof}[Proof of Lemma~\ref{lem:qwh232}]
  First note a few obvious facts:
  For all $n$-ml-creatures $\cc$, we have
  \begin{equation}\label{eq:numPOSS}
      \left|\POSS_{n,\supp(\cc)}\right|\leq  \fmax(n-1)^{n\maxsupp(n)}
  \end{equation}
  and for a condition $p$ we get, according to~\ref{def:initialsequence}(\ref{item:maxposs}),
  \begin{equation}\label{eq:numposs}
      \left|\poss(p,n)\right|\leq  \fmax(n-1)^{n\maxsupp(n-1)}<\maxposs(n),
  \end{equation}
  According to~\ref{def:initialsequence}(\ref{item:maxsupp}), we get:
  If $|\supp(\cc)| \geq \maxsupp(n)/2$, then
  \begin{equation}\label{eq:doublesuppok}
    \nor(\cc)\leq
      1/\maxposs(n)\, \log_{2}\left( \maxnor(n) - \log_2(\maxsupp(n)) +1 
      \right) =0.
  \end{equation}

  {\bf The local $\Delta$ property:}
  We only have to check that ``taking the union of identical creatures
  with disjoint domains'' decreases the norm by at most one, the rest is 
  just notation:

  Given an $n$-ml creature $(u^\cc, \bar w^\cc, d^\cc)$ and 
  an enumeration $\bar i:|u^\cc|\to u^\cc$, 
  we define the local type to contain the following information
  for $m,m'<|u^\cc|$:
    $d^\cc$, $|u^\cc|$, whether $\bar i(m)\in\mu$, whether 
    $\varepsilon(\bar i(m))=\bar i(m')$, and 
    the sequence of the atomic creatures (enumerated by $\bar i$).%
\footnote{%
    More formally: the sequences $(w^\cc_{\bar i(m)})_{m<|u|,\bar i(m)\in\mu}$ and
    $(w^\cc_{\bar i(m),k})_{m<|u|,\bar i(m)\notin\mu,k\in
    A^\cc_{\varepsilon(\bar i(m))}}$.%
}
  Take $\cc_1$ and $\cc_2$ as in the Definition~\ref{def:localdelta} of
  the local $\Delta$ property. 
  Since $\nor(\cc_1)>1$, we know by~\eqref{eq:doublesuppok} that
  $|\supp(\cc)|<\maxsupp(n)/2$. So we can
  define the $n$-ml-creature $\cd$ by 
  $d^\cd=d^{\cc_1}=d^{\cc_2}$;
  $u^\cd=u^{\cc_1}\cup u^{\cc_2}$;  and
  for $\epsilon\in \mu$ we set $w^\cd_\epsilon$ to
  be  $w^{\cc_1}_\epsilon$ or $w^{\cc_2}_\epsilon$,
  whichever is defined (if both are defined, they have
  to be equal, since the type is the same); and in the 
  same way we define $w^\cd_{\alpha,k}$ for $\alpha\in I_\epsilon$
  and $k\in A^\cd_\epsilon$.

  As already mentioned, the only thing we have to check is
  that $\nor(\cd)\geq \nor(\cc)-1$ (for $\cc=\cc_1$ or $\cc=\cc_2$,
  which does not make any difference).
  Since $\cd$ consists of the same atomic creatures as
  $\cc$, we get $\minnor(\cd)=\minnor(\cc)$, and therefore
  \begin{align*}
    \nor(\cd)
      &\geq 
        1/\maxposs(n)\,\log_{2}\left(
        \minnor(\cd)
        -\log_2(2 |\supp(\cc)|)-d\right)
      \\
      &\geq   
        1/\maxposs(n)\,\log_{2}\left( \left(
        \minnor(\cc)
        -\log_2(|\supp(\cc)|)-d\right) /2\right)
      \\
      &= \nor(\cc)-1/\maxposs(n).
  \end{align*}

  {\bf The domain of the generic:}
  Given $\alpha\in I^*$, we can just enlarge any $n$-ml-creature
  creature $\cc=(u^\cc,\bar w^\cc,d^\cc)$ in the following
  way: Increase the domain by
  $\alpha$ and (if $\alpha\notin \mu$) additionally by $\varepsilon(\alpha)$,
  and pick for the new positions atomic creatures with norm $\maxnor(n)$.
  The same argument as for the local $\Delta$-property shows
  that the norm of the new creature decreases by at most $1/\maxposs(n)$.
  So we can modify any condition to a stronger condition with a 
  domain containing $\alpha$
  (as in Lemma~\ref{lem:genericdomain}).

  {\bf Halving:}
  Halving follows directly from the definition of the norm:
  Given $\cc=(u^\cc,\bar w^\cc,d^\cc)$, set 
  $\half(\cc)=(u^\cc,\bar w^\cc,d')$ with
  \[
    d'=d^\cc+1/2\,\left[ \minnor(\cc) - \log_2(\supp(\cc)) -d^\cc \right].
  \]
  Fix $\cd=(u^\cd,\bar w^\cd,d^\cd)\in\cS(\half(\cc))$ (so in particular,
  $d^\cd\geq d'$). We can unhalve
  $\cd$ to $\tilde \cd=(u^\cd,\bar w^\cd,d^\cc)$.
  Straightforward calculations show that the halving
  properties are satisfied. In particular: If $\nor(\cd)>0$, then
  \[
    \minnor(\cd)-\log_2(\supp(\cd))-d^\cd>1.
  \]
  To calculate $\nor(\tilde \cd)$, we use
  \begin{align*}
    \minnor(\cd)-\log_2(\supp(\cd))-d^\cc
    &> 1+d^\cd-d^\cc \geq 1+d'-d^\cc> \\
    &>1/2\,\left[ \minnor(\cc) - \log_2(\supp(\cc)) -d^\cc \right].
  \end{align*}
  So  $\nor(\tilde \cd)\geq \nor(\cc)-1/\maxposs(n)$.

  {\bf Bigness:}
    Let $\cc$ be an $n$-ml-creature. Set 
   $B\coloneqq \fmax(n-1)^{n\fmax(n-1)}$.
    To show $(B,1)$-bigness, we pick some
    $G:\POSS_{n+1,\supp(\cc)}\to B$, and 
    we have to find  a $\cd\in \cSp^{1}(\cc)$ 
    such that $G$ only depends on $\eta\restriction n$.
    (More formally: there is 
    a $G':\POSS_{n,\supp(\cc)}\to B$ such that
    $G(\eta)=G'_0(\nu)$ for all $\eta\in \cd[\nu]$.)

    Set $S=\POSS_{n,\supp(\cc)}$ and 
    $M=\prod_{\epsilon\in\supp(\cc)\cap \mu} A^\cc(\epsilon)$.
    ($S$ and $M$ stand for ``small'' and ``medium'', respectively.)
    Note that according to~\eqref{eq:numPOSS} and~\ref{def:initialsequence}(\ref{item:gmin}),
    \begin{equation}\label{eq:bblubb}
      |S\times M|\leq 
      \fmax(n-1)^{n\maxsupp(n)} \cdot k^*(n)^{\maxsupp(n)}< \gmin(n).
    \end{equation}

    If we fix $\eta\in S$ and $x\in M$,
    then $G$ can be written as a function from 
    $\prod_{\alpha\in\supp(\cc)\setminus \mu}
    A^{\cc}_{\alpha,x(\varepsilon(\alpha))}$  to $B$.

    We get:
    \begin{itemize}
      \item All the atomic creatures involved are $\gmin(n)$-nice.
      \item $|\supp(\cc)\setminus \mu|<\maxsupp(n)<\gmin(n)$.
      \item $B<2^{\gmin(n)}$.
    \end{itemize}
    So we can apply 
    Fact~\ref{facts:214f}(\ref{item:atomlem1}) and get
    successors $v_\alpha\in
    \cSp^{1/\gmin(n)}(w^{\cc}_{\alpha,x(\varepsilon(\alpha))})$ such that $G$
    is constant (with respect to the new creatures).

    We can iterate this for all $(\eta,x) \in S\times M$,
    each time decreasing the norm of some of the atomic creatures
    on $\supp(\cc)\setminus \mu$ by at most $1/\gmin(n)$. By~\eqref{eq:bblubb},
    in the end we get 
    $v_{\alpha,k}\in \cSp^{1}(w^{\cc}_{\alpha,k})$ 
    for all $\alpha\in u^\cc\setminus \mu$ 
    and $k\in A^\cc_{\varepsilon(\alpha)}$ 
    such that (modulo these new creatures)
    $G$ only depends on $(\eta,x)\in S\times M$;
    or, in other words, $G$ can be written as 
    function fomr $M$ to $B^S$.

    It remains to get rid of the dependence on $M$.
    For this, just note that all the atomic creatures
    $w^\cc_\epsilon$ (for $\epsilon\in u^\cc\cap \mu$)
    are $\Bmin(n)$-nice, $\maxsupp(n)<\Bmin(n)$ and $\Bmin(n)>B^S$,
    so we can find successors on which $G$ is constant.
\end{proof}

\subsection{Proof of the main theorem}
\begin{Def}
    \begin{itemize}
      \item $\nu_i\coloneqq \nugen\restriction \{i\}$ for all $i\in I^*$.
        (We interpret $\nu_i$ as function from $\omega$ to $\omega$.)
      \item $f_\epsilon(n)\coloneqq f_{n,\nu_\epsilon(n)}$ for 
        $\epsilon\in\mu$,
        and analogously for $g_\epsilon$.
      \item $\mycfa_\epsilon\coloneqq \mycfa_{f_\epsilon,g_\epsilon}$ for $\epsilon\in\mu$,
            and analogously for $\myc_\epsilon$.
    \end{itemize}
\end{Def}

So $Q_\mathfrak p$ forces that 
$\nu_\epsilon(n)<k^*(n)$ for all $n\in\omega$, and that
$\nu_\alpha(n) < f_\epsilon(n)$ for all but
finitely many $n$. (There might be finitely many exceptions, since the initial
trunk at $\alpha$ might not fit to the initial trunk at
$\varepsilon(\alpha)$.)

To prove the main theorem, it is enough to show the following:
\begin{quote}
  $Q_\mathfrak p$ forces $2^{\al0}=\mu$ 
  and $\myc_\epsilon=\mycfa_\epsilon=\kappa_\epsilon$
  for all $\epsilon\in \mu$.
\end{quote}
This will be done in Lemmas~\ref{lem:sizeofcont},~\ref{lem:cfalegkappa}
and~\ref{lem:cexgeqkappa}.

\begin{Lem}\label{lem:sizeofcont}
  $Q_\mathfrak p$ forces $2^{\al0}= \mu$.
\end{Lem}
\begin{proof}
  First note that trivially all $\nu_i$ are different: 
  Fix   $p\in Q_\mathfrak p$ and $i\neq j$ in $I^*$.
  We already know that $Q_\mathfrak p$ forces that the domain of the generic
  is $\omega\times I^*$, in particular we can assume that $i,j\in\dom(p)$.
  Choose $n$ so that $\nor(p,n)>1$.
  In particular, all the atomic creatures involved have
  norm bigger than 1 and therefore more than one possible value. 
  So we can choose an $\eta\in\poss(p,n+1)$ such that
  $\eta(n,i)\neq \eta(n,j)$. Then $p\wedge \eta$ 
  forces $\nu_i\neq \nu_{j}$.
  
  This shows that the continuum has size at least $\mu$ in the extension.
  
  Due to continuous reading of names, every real $r$ in the extension
  corresponds to a condition $p$ in $Q_\mathfrak p$ together with 
  a continuous way to read $r$ off $p$.

  More formally: For each $n\in\omega$
  there are 
  $h(n)\in\omega$ and a function $\eval(n):\poss(p,h(n))\to \omega$  such
  that $p\wedge \eta$ forces $r(n)=\eval(n)(\eta)$
  for all $\eta\in \poss(p,h(n))$.

  Since there are only $\mu^{\al0}=\mu$ many such pairs of conditions
  and continuous readings, there can be at most $\mu$ many reals in the
  extension.
\end{proof}

We now mention a simple but useful property of the atomic creatures:

\begin{Lem}\label{lem:disjointvalue}
  Assume $w_1$ and $w_2$ are two atomic creatures that appear in some $n$-ml-creature
  $\cc$. Then there are $v_i\in \cSp^{2/\Bmin(n)}(w_i)$ (for $i\in\{0,1\}$)
  such that $\val(v_0)\cap \val(v_1)=\emptyset$.
\end{Lem}

\begin{proof}
  Apply decisiveness to get
  successors $w^s$ of $w_1$ and $w^b$ of $w_2$ (or the
  other way round) such that the norms decrease by at most $1/\Bmin(n)$
  and 
  $|\val(w^s)|<K$ and $w^b$ is hereditarily $K+1$-big for some $K\in\omega$.

  [In more detail: Since $w_0$ is decisive, there is a natural number
  $K$ such that there is a $K$-small successor $w^s_0$ as well as a
  $K$-big successor $w^b_0$ of $w_0$.  On the other hand,
  again using decisiveness, 
  $w_1$ has a successor $w'_1$
  that is either $K$-small (then we set 
  $w^s=w'_1$ and $w^b=w^b_0$)
  or $K$-big (then we set $w^b=w'_1$ and
  $w^s=w^s_0$).]

  Enumerate $\val(w^s)$, and define $G$ from $\val(w^b)$ to $K+1$ 
  as follows: If $l\in \val(w^b)$ is the $k$-th element of
  $\val(w^s)$, set $G(l)=k+1$. Otherwise, set $G(l)=0$.

  Using $K+1$-bigness, we get a $G$-homogeneous successor
  $v$ of $w^b$. Then $v$ and $w^s$ are as required.
\end{proof}

A simple application of this Lemma gives us ``seperated support'':
\begin{Lem}\label{lem:seperatedsupport}
      For $p\in Q_\mathfrak p$ there is a $q\leq p$ such that 
      $q(n)\in \cSp^1(p(n))$ for all $n\geq \trnklg(q)$ and 
      $A^{q(n)}_{\epsilon_0} \cap A^{q(n)}_{\epsilon_1}=\emptyset$ for all
      $n$ and $\epsilon_0\neq \epsilon_1$ in $\supp(q,n)\cap \mu$.
\end{Lem}

\begin{proof}
  Fix $n$ and a pair
  $\epsilon_0\neq \epsilon_1$ in $\supp(p,n)\cap \mu$.
  According to Lemma~\ref{lem:disjointvalue}, we
  can find $v_{\epsilon_i}\in \cSp^{2/B(n)}(w^{p(n)}_{\epsilon_i})$ for
  $i\in\{0,1\}$
  with disjoint values. Iterate this for all pairs in $\supp(p,n)\cap \mu$
  (note that there are less than $\maxsupp(n)^2<\Bmin(n)/2$ many, according
  to~\ref{def:initialsequence}(\ref{item:bmin})).
\end{proof}

\begin{Lem}
  Fix $\epsilon_0\in \mu$. Then $Q_\mathfrak p$ forces
  that $\mycfa_{\epsilon_0}\leq \kappa_{\epsilon_0}$.
\end{Lem}

\begin{proof}\label{lem:cfalegkappa}
  Set $I'=\{\epsilon_0\}\cup I_{\epsilon_0}$.
  We will show that in the $Q_\mathfrak p$ extension of $V$
  the family of those $(f_{\epsilon_0},g_{\epsilon_0})$-slaloms 
  that can (in $V$) be read continuously from $I'$ alone
  form a $\forall$-cover.
  This proves the
  Lemma, since there are
  only $\kappa_{\epsilon_0}^{\al0}=\kappa_{\epsilon_0}$ many 
  continuous readings on $I'$.

  Assume that $r$ is a name for an element of $\prod f_{\epsilon_0}$.
  Fix $p\in Q_\mathfrak p$. Using Corollary~\ref{cor:gurk},
  without loss of generality we can assume 
  that $p$ rapidly reads $r$ (i.e., $r\restriction n$ is $n$-decided by
  $p$) and that it satisfies seperated support as in the previous Lemma.

  We will construct a $q\leq p$ and a name for an $(f_{\epsilon_0},g_{\epsilon_0})$-slalom
  $Y$ that can be continuously read from $q\restriction I'$
  such that $q$ forces $r(n)\in Y(n)$ for all but finitely many $n\in\omega$.
  (This proves the Lemma.)

  Fix $n_0$ such that $\nor(p,n)>2$ for all $n\geq n_0$ and
  set $q(n)=p(n)$ for $n<n_0$. We
  construct $Y(n)$ and $q(n)$ by induction on $n\geq n_0$.
  We set
  $\supp(q,n)\coloneqq\supp(p,n)$ and $\trnk(q)\coloneqq \trnk(p)$.
  I.e., the supports and trunks do not change at all.
  So by induction $\poss(q,n)\subseteq \poss(p,n)$.

  Let us denote the $n$-ml-creature $p(n)$ by $\cc$. 
  We have to define the $n$-ml-creature $q(n)$ (let us call it $\cd$)
  with $u^\cd=u^\cc$ (call it $u$). We set $d^\cd\coloneqq d^\cc$.
  On $\mu$, we do not change anything: 
  For $\epsilon\in u\cap \mu$ we set
  $w^\cd_\epsilon\coloneqq w^\cc_\epsilon$ (call it $w_\epsilon$,
  and set $A_\epsilon\coloneqq \val(w_\epsilon)=A^\cc_\epsilon=A^\cd_\epsilon$). 
  It remains to define 
  $w^\cd_{\alpha,k}\in \cSp^1(w^\cc_{\alpha,k})$ for $\alpha\in u\cap I_\epsilon$
  and $k\in w_\epsilon$.
  Then, since the norms of all the atomic creatures only decrease
  by 1, we know that 
  $\nor(\cd)$ will definitely be bigger than $\nor(\cc)-1$, as required.

  Let $T$ (for ``trunk'') be the set of pairs $(\eta,x)$ such that
  $\eta\in \poss(q,n)$ and $x\in \prod_{\epsilon\in u\cap \mu}
  A_\epsilon$.
  \begin{equation}\label{eq:bblubb5}
    |T|\leq \gmin(n).
  \end{equation}

  We now partition $\supp(\cc)\setminus \mu$ into 
  sets called $S$, $M$, $L$ (small, medium, large):
  Set $M=\supp(\cc)\cap I_{\epsilon_0}$. 
  Using seperated support, we know that 
  every $\epsilon\neq \epsilon_0$ in $u\cap \mu$ satisfies either
  $x(\epsilon)<x(\epsilon_0)$ (then we put all elements of $I_\epsilon\cap u$
  into $S$) or $x(\epsilon)>x(\epsilon_0)$ (then we put them into $L$).

  Rapid reading implies that (modulo the pair $(\eta,x)$) the natural
  number 
  $r(n)$ can be interpreted as function 
  \[
    r(n):\prod_S\times \prod_M \times \prod_L 
    \ \to \ 
    f_{n,x(\epsilon_0)}.
  \]
  where we set (for $X\in\{S,M,L\}$)
  \[
    \prod_X \coloneqq \prod_{\alpha\in X} A_{\alpha,x(\varepsilon(\alpha))}.
  \]
  Our goal is to get a name $Y(n)$ for a small subset 
  of $f_{n,x(\epsilon_0)}$ that only depends on $M$ and 
  and contains $r(n)$.
  
  First note that we can rewrite $r(n)$ as function 
  \[
    r(n)\ : \ \prod_L
    \ \to \
    f_{n,x(\epsilon_0)}^{\Pi_S\times \Pi_M}.
  \]
  Using the fact that the
  atomic creatures in $L$ are nice enough,%
  \footnote{%
    they all satisfy $g_{n,x(\epsilon_0)+1}$ niceness, and
    in~\ref{def:initialsequence}(\ref{item:g}) we assumed that
    $g_{n,x(\epsilon_0)+1}$ is bigger than $f_{n,x(\epsilon_0)}^{\prod_S\times \prod_M}$,
    since $\prod S\times \prod M$ has size less than $f_{n,x(\epsilon_0)}^{k^*(n)}$.
   Now use Fact~\ref{facts:214f}(\ref{item:atomlem1}). So the norms decrease
   at most by $1/g_{n,x(\epsilon_0)+1}<1/\gmin(n)$.
  }
  we can find
  successors of these creatures that evaluate $r(n)$ 
  to a constant value, and such that the norms decrease by less than
  $1/\gmin(n)$.
  We define $w'_{\alpha,x(\epsilon)}$ to be these successors
  for $\epsilon\in L$; and leave the other 
  atomic creatures unchanged. Now for every $y\in \Pi_M$
  there are only
  $|\Pi_S|$ many possible values for $r(n)$, call this sets of 
  possible values 
  $Y(\eta,x,y)$. 

  Iterate this procedure for all pairs $(\eta,x)\in T$. The same
  atomic creature may be decreased more than once, but at
  most $\gmin(n)$ many times, according to~\eqref{eq:bblubb5}. So 
  in the end, the norms of the 
  resulting  atomic creatures decrease by less than $1$.
  This finishes the definition of $q(n)$.

  We still have to define $Y(n)$ as a function from
  the possible values $(k_0,y_0)$ on $\{\epsilon_0\}\cup I_{\epsilon_0}$,
  i.e., as a function with domain $\{(k_0,y_0):\, k_0\in A_{\epsilon_0},
  y_0\in\prod_{\alpha\in I_{\epsilon_0}} A^\cd_{\alpha,k_0}\}$.
  We set $Y(n)$
  to be $\bigcup_{(\eta,x)\in T,\ x(\epsilon_0)=k_0 }\tilde Y(\eta,x,y_0)$.
  This set has size less than $g_{n,k_0}$, as required.%
\footnote{%
  $|Y(\eta,x,y)|\leq |\Pi_S|$,
  so $|Y(n)|\leq |T\times \Pi_S|\leq \maxposs(n)\cdot k^*(n)^{\maxsupp(n)}\cdot f_{n,k_0-1}^{\maxsupp(n)}$,
  which is smaller than $g_{n,k_0}$ according to~\ref{def:initialsequence}(\ref{item:g}).
}
\end{proof}

\begin{Lem}\label{lem:hilfslemiz612}
  Fix $|J|\leq \maxsupp(n)$ and for each $i\in J$
  an atomic creature $w_i$ that is
  $(\maxsupp(n),1/\gmin(n))$-decisive.
  Then there are $w'_i\in\cSp^{1/{k^*(n)}}(w_i)$ for all $i\in J$ 
  and a linear order $\leq_J$ on $J$ such that each $w'_i$
  is hereditarily $\prod_{j<_J i} |\val(w_i)|$ big.
\end{Lem}

\begin{proof}
  For any $i\in J$, apply  decisiveness to the atomic creature $w_i$.
  This gives some $K_i$ and a $K_i$-big as well as a $K_i$-small successor
  of $w_i$. 
  Pick the $i$ with a minimal $K_i$, let this $i$ be the first 
  element of the $<_J$-order, set $w'_i$ to be the $K_i$-small
  successor, and pick for all other $j$ the $K_j$-big successor.
  Repeat this construction for $J\setminus \{i\}$.

  So in the end we order the whole set $J$, decreasing each creature
  at most $\maxsupp(n)$ many times by at most $1/\gmin(n)$.
\end{proof}

It remains to be shown:
\begin{Lem}\label{lem:cexgeqkappa}
  $Q_\mathfrak p$ forces that $\myc_{\epsilon_0}\geq \kappa_{\epsilon_0}$.
\end{Lem}

\begin{proof}
  Note that it is forced that $f_{\epsilon_0}(n)/g_{\epsilon_0}(n)$ converges
  to infinity, therefore (by the usual diagonalization) it is forced that
  $\myc_{\epsilon_0}>\al0$. So if $\kappa_{\epsilon_0}=\al1$ there is nothing to do.

  So assume that $\al1\leq \lambda<\kappa_{\epsilon_0}$ and assume towards a contradiction
  that some $p_0$ forces
  $\{Y_\zeta:\, \zeta\in\lambda \}$ is an $\exists$-cover.

  For each $\zeta\in\lambda$ we
  can find a maximal antichain $A_\zeta$ below $p_0$ 
  such that every condition in $A_\zeta$ rapidly reads $Y_\zeta$.
  Let $D$ be the union of the domains of all elements of any of the $A_\zeta$ for $\zeta\in\lambda$.
  Due to $\al2$-cc, $D$ has size 
  $\al0 \times \al1 \times \lambda = \lambda$ which is less than $\kappa_{\epsilon_0}$.
  So we can pick a
  $\beta\in I_{\epsilon_0}\setminus D$ and a
  $p_1\leq p_0$ deciding the $Y_\zeta$ that $\exists$-covers $\nu_\beta$. From now
  an, we will call $Y_\zeta$ just $Y$. Pick some $p\leq p_1$ 
  that is stronger than some element of $A_\zeta$. To summarize:
  \begin{gather*}
    \text{$p$ restricted to $\dom(p)\setminus \{\beta\}$ rapidly reads $Y$. 
    (I.e., $Y$ does not depend on the values at $\beta$.)}
    \\
    p\text{ forces that }Y(n) \text{ is a subset of }f_{\epsilon_0}(n) \text{ of size less than } g_{\epsilon_0}(n) \text{ for all }n,
    \\
    p\text{ forces that there are infinitely many }n\text{ such that }\nu_\beta(n)\in Y(n)
  \end{gather*}

  We will now derive the desired contradiction: 
  We will find an $n_0\in\omega$ and a
  $q\leq p$ forcing that $\nu_\beta(n)\notin Y(n)$ for all $n\geq n_0$.

  Pick $n_0$ such that $\nor(p,n)>2$ for all $n\geq n_0$.
  We will construct $q(n)=:\cd$ by induction on $n\geq n_0$.
  Denote $p(n)$ by $\cc$.
  We set $\trnk(q)\coloneqq \trnk(p)$
  and $u^\cd\coloneqq u^\cc$ (call it $u$),
  so the supports and the trunks do not change at all, and 
  by induction $\poss(q,n)\subseteq \poss(p,n)$.
  We also set $d^\cd\coloneqq d^\cc$.  On $\mu$, nothing changes:
  For 
  $\epsilon\in u\cap \mu$ set
  $w^\cd_\epsilon\coloneqq w^\cc_\epsilon$ (call it $w_\epsilon$,
  and set $A_\epsilon=\val(w_\epsilon)=A^\cc_\epsilon=A^\cd_\epsilon$).

  It remains to construct $w^\cd_{\alpha,k}\in \cSp^1(w^\cc_{\alpha,k})$
  for $\alpha\in u\cap I_\epsilon$
  and $k\in A_{\epsilon}$. 

  Let $T$ (for ``trunk'') consist of all pairs $(\eta,x)$
  such that $\eta\in \poss(q,n)$ and
  $\eta\in\prod_{\epsilon\in u\cap \mu} A_\epsilon$.
  Note that $|T|$ is smaller than $\gmin(n)$, as already stated in~\eqref{eq:bblubb5}.

  Given $(\eta,x)$ in $T$, we apply the previous Lemma
  to $J\coloneqq  u\setminus \mu$ 
  and 
  the sequence 
  $(w^\cc_{\alpha,x(\varepsilon(\alpha))})_{\alpha  \in J}$.
  This gives us successor creatures $(w'_\alpha)_{\alpha\in J}$
  as well as an order $<_J$ of $J$.
  Partition $J$ into $S=\{i<_J \beta\}$, $\{\beta\}$, and 
  $L=\{i>_J\beta\}$.

  So (given $\eta$ and $x$), we can write $Y(n)$ (which does
  not depend on $\beta$) as function
  from $\prod_{\alpha\in L}\val(w'_\alpha)
  \times \prod_{\alpha\in S} \val(w'_\alpha)$
  to the family of subsets of $f_{n,x(\epsilon_0)}$ of size less than
  $g_{n,x(\epsilon_0)}$. Therefore we can use bigness
  to once more
  strenghen the atomic creatures indexed by $L$ and thus remove 
  the dependence of $Y(n)$ from $L$.
  We now take the union $\tilde Y$ of all the remaining possibilities for $Y(n)$
  and get a set of size less than
  $g_{n,x(\epsilon_0)} \cdot | \prod_{\alpha\in S} \val(w'_\alpha)|$,
  which is smaller than the bigness of 
  $w'_{\beta}$. So (just as in the proof of
  Lemma~\ref{lem:disjointvalue}) 
  we can strengthen this creature $w'_{\beta}$ 
  to be disjoint to $\tilde Y$.

  As usual, we now iterate this construction for all pairs $(\eta,x)\in T$.
  The resulting $n$-ml-creature $q(n)$ guarantees that $\nu_\beta(n)$
  is not in $Y(n)$, as required.
\end{proof}

\bibliographystyle{plain}
\bibliography{961}

\end{document}